\documentclass[english]{article}
\usepackage[T1]{fontenc}
\usepackage[latin9]{inputenc}
\usepackage[letterpaper]{geometry}
\usepackage{textcomp}
\usepackage{amsthm}
\usepackage{amsmath}
\usepackage{graphicx}
\usepackage{amssymb}
\usepackage[authoryear]{natbib}

\makeatletter
\newcommand{\lyxaddress}[1]{
\par {\raggedright #1
\vspace{1.4em}
\noindent\par}
}
\theoremstyle{plain}
\newtheorem{thm}{Theorem}
  \theoremstyle{definition}
  \newtheorem{defn}[thm]{Definition}
  \theoremstyle{plain}
  \newtheorem{prop}[thm]{Proposition}
  \theoremstyle{remark}
  \newtheorem{rem}[thm]{Remark}
 \theoremstyle{definition}
  \newtheorem{example}[thm]{Example}

\DeclareMathOperator{\interior}{int}

\makeatother

\usepackage{babel}

\begin{document}

\title{Coherent frequentism}

\author{David R. Bickel}

\maketitle

\lyxaddress{Ottawa Institute of Systems Biology; Department of Biochemistry,
Microbiology, and Immunology; Department of Mathematics and Statistics
\\
University of Ottawa; 451 Smyth Road; Ottawa, Ontario, K1H 8M5}

\lyxaddress{+01 (613) 562-5800, ext. 8670; dbickel@uottawa.ca}
\begin{abstract}
By representing the range of fair betting odds according to a pair
of confidence set estimators, dual probability measures on parameter
space called frequentist posteriors secure the coherence of subjective
inference without any prior distribution. The closure of the set of
expected losses corresponding to the dual frequentist posteriors constrains
decisions without arbitrarily forcing optimization under all circumstances.
This decision theory reduces to those that maximize expected utility
when the pair of frequentist posteriors is induced by an exact or
approximate confidence set estimator or when an automatic reduction
rule is applied to the pair. In such cases, the resulting frequentist
posterior is coherent in the sense that, as a probability distribution
of the parameter of interest, it satisfies the axioms of the decision-theoretic
and logic-theoretic systems typically cited in support of the Bayesian
posterior. Unlike the \emph{p}-value, the confidence level of an interval
hypothesis derived from such a measure is suitable as an estimator
of the indicator of hypothesis truth since it converges in sample-space
probability to 1 if the hypothesis is true or to 0 otherwise under
general conditions.
\end{abstract}
Keywords: attained confidence level; coherence; coherent prevision;
confidence distribution; decision theory; minimum expected loss; fiducial
inference; foundations of statistics; imprecise probability; maximum
utility; observed confidence level; problem of regions; significance
testing; upper and lower probability; utility maximization

\newpage{}

\section{\label{sec:Introduction}Introduction}

\subsection{\label{sub:Background}Motivation}

A well known mistake in the interpretation of an observed confidence
interval confuses \emph{confidence }as a level of certainty with {}``confidence''
as the \emph{coverage rate}, the almost-sure limiting rate at which
a confidence interval would cover a parameter value over repeated
sampling from the same population. This results in using the stated
confidence level, say 95\%, as if it were a probability that the parameter
value lies in the particular confidence interval that corresponds
to the observed sample. A practical solution that does not sacrifice
the 95\% coverage rate is to report a confidence interval that matches
a 95\% \emph{credibility interval} computable from Bayes's formula
given some \emph{matching prior} distribution \citep{RefWorks:1222}.
In addition to canceling the error in interpretation, such matching
enables the statistician to leverage the flexibility of the Bayesian
approach in making jointly consistent inferences, involving, for example,
the probability that the parameter lies in any given region of the
parameter space, on the basis of a posterior distribution firmly anchored
to valid frequentist coverage rates. Priors yielding exact matching
of predictive probabilities are available for many models, including
location models and certain location-scale models \citep{RefWorks:1155,RefWorks:1154}.
Although exact matching of fixed-parameter coverage rates is limited
to location models \citep{RefWorks:1158,RefWorks:61}, priors yielding
asymptotic matching have been identified for other models, e.g., a
hierarchical normal model \citep{RefWorks:1155}. For mixture models,
all priors that achieve matching to second order necessarily depend
on the data but asymptotically converge to fixed priors \citep{RefWorks:1234}.
Data-based priors can also yield second-order matching with insensitivity
to the sampling distribution \citep{RefWorks:1236}. Agreeably, \citet{FraserReidYi2008}
suggested a data-dependent prior for approximating the likelihood
function integrated over the nuisance parameters to attain accurate
matching between Bayesian probabilities and coverage rates. These
advances approach the vision of building an objective Bayesianism,
defined as a {}``universal recipe for applying Bayes theorem in the
absence of prior information'' \citep{RefWorks:193}. 

Viewed from another angle, the fact that close matching can require
resorting to priors that change with each new observation, cracking
the foundations of Bayesian inference, raises the question of whether
many of the goals motivating the search for an objective posterior
can be achieved apart from Bayes's formula. It will in fact be seen
that such a probability distribution lies dormant in nested confidence
intervals, securing the above benefits of interpretation and coherence
without matching priors, provided that the confidence intervals are
constructed to yield reasonable inferences about the value of the
parameter for each sample from the available information. 

Unless the confidence intervals are conservative by construction,
the condition of adequately incorporating any relevant information
is usually satisfied in practice since confidence intervals are most
appropriate when information about the parameter value is either largely
absent or included in the interval estimation procedure, as it is
in random-effects modeling and various other frequentist shrinkage
methods. Likewise, confidence intervals known to lead to pathologies
tend to be avoided. (Pathological confidence intervals often emphasized
in support of credibility intervals include formally valid confidence
intervals that lie outside the appropriate parameter space \citep{RefWorks:156}
and those that can fail to ascribe 100\% confidence to an interval
deduced from the data to contain the true value \citep{RefWorks:180}.)
A game-theoretic framework makes the requirement more precise: for
the 95\% confidence interval to give a 95\% degree of certainty in
the single case and to support coherent inferences, it must be generated
to ensure that, on the available information, 19:1 are approximately
fair betting odds that the parameter lies in the observed interval.
This condition rules out the use of highly conservative intervals,
pathological intervals, and intervals that fail to reflect substantial
pertinent information. In relying on an observed confidence interval
to that extent, the decision maker ignores the presence of any recognizable
subsets \citep{RefWorks:1137}, not only slightly conservative subsets,
as in the tradition of controlling the rate of Type I errors \citet{ConditionallyAcceptableRecenteredSetEstimators1987},
but also slightly anti-conservative subsets. Given the ubiquity of
recognizable subsets \citep{RefWorks:1291,RefWorks:1239}, this strategy
uses pre-data confidence as an approximation to post-data confidence
in the sense in which expected Fisher information approximates observed
Fisher information \citep{RefWorks:1321}, aiming not at exact inference
but at a pragmatic use of the limited resources available for any
particular data analysis. Certain situations may instead call for
careful applications of conditional inference \citep{GoutisCasella1995b,Sundberg2003299,FraserAncillaries2004a}
for basing decisions more directly on the data actually observed.

\subsection{Direct inference and attained confidence}

The above betting interpretation of a frequentist posterior will be
generalized in a framework of decision to formalize, control, and
extend the common practice of equating the level of certainty that
a parameter lies in an observed confidence interval with the interval
estimator's rate of coverage over repeated sampling.

Many who fully understand that the 95\% confidence interval is defined
to achieve a 95\% coverage rate over repeated sampling will for that
reason often be substantially more certain that the true value of
the parameter lies in an observed 99\% confidence interval than that
it lies in a 50\% confidence interval computed from the same data
\citep[pp. 11-12]{RefWorks:999,RefWorks:170}. This \emph{direct inference},
reasoning from the frequency of individuals of a population that have
a certain property to a level of certainty about whether a particular
sample from the population, is a notable feature of inductive logic
\citep[e.g.,  ][]{RefWorks:999,Jaeger2005b} and often proves effective
in everyday decisions. Knowing that the new cars of a certain model
and year have speedometer readings within 1 mile per hour (mph) of
the actual speed in 99.5\% of cases, most drivers will, when betting
on whether they comply with speed limits, have a high level of certainty
that the speedometer readings of their particular new cars of that
model and year accurately report their current speed in the absence
of other relevant information. (Such information might include a reading
of 10 mph when the car is stationary, which would indicate a defect
in the instrument at hand.) If the above betting interpretation of
the confidence level holds for an interval given by some predetermined
level of confidence, then coherence requires that it hold equally
for a level of confidence given by some predetermined hypothesis.

Fisher's fiducial argument also employed direct inference (\citealp{FisherLogicalInversion1945};
\citealp[pp. 34-36, 57-58]{RefWorks:985}; \citealp[Chapter 9]{RefWorks:305};
\citealp{RefWorks:1179}). The present framework departs from his
in its applicability to inexact confidence sets, in the closer proximity
of its probabilities to repeated-sampling rates of covering vector
parameters, in its toleration of reference classes with relevant subsets,
and in its theory of decision. Since the second and third departures
are shared with recent methods of computing the confidence probability
of an arbitrary hypothesis (§\ref{sub:Hypothesis}), the main contribution
of this paper is the general framework of inference that both motivates
such methods given an exact confidence set and extends them for use
with approximate, valid, and nonconservative set estimators and for
coherent decision making, including prediction and point estimation. 

This framework draws from the theory of coherent upper and lower probabilities
for the cases in which no exact confidence set with the desired properties
is available. To allow indecision in light of inconclusive evidence,
these non-additive probabilities have been formulated for lotteries
in which the agent may either place a bet or refrain from betting
or, equivalently, in which the casino posts different odds to be used
depending on whether a gambler bets for or against a hypothesis. Confidence
decision theory will be formulated for this scenario by setting an
agent's prices of buying and selling a gamble on the hypothesis that
a parameter $\theta$ is in some set $\Theta^{\prime}\in\Theta$ according
to the confidence levels of a valid set estimate and a nonconservative
confidence set estimate that coincide with $\Theta^{\prime}.$ As
a result, the hypothesis has an interval of confidence levels rather
than a single confidence level. Equating the buying and selling prices
reduces the upper and lower probability functions to a single frequentist
posterior, a probability measure on parameter space $\Theta,$ and
thus reduces the interval to a point.

\subsection{Overview}

This subsection outlines the organization of the remainder of the
paper while offering a brief summary. 

After preliminary concepts are defined (§\ref{sub:Preliminaries}),
Section \ref{sub:Confidence-measures-etc} presents the new framework
for confidence-based inference and decision. The family of probability
measures (frequentist posteriors) used in inference and decision can
be stated in terms of coherent lower and upper probabilities and is
thus completely self-consistent according to a widely accepted account
of coherence derived from ideas of Bruno de Finetti (§\ref{sub:Coherence}).
This lays a foundation for decisions and for flexible inference about
the truth of hypotheses without invoking the likelihood principle
(§§\ref{sub:Decisions-given-confidence}, \ref{sub:Likelihood}).
The framework is compared to other versions of frequentist coherence
based on upper and lower probabilities in Section \ref{sub:Likelihood}.

While reporting an interval level of confidence in a hypothesis has
the advantage of honestly communicating the insufficiency of the data
to determine a single confidence level, such intervals are less useful
in situations requiring the automation of decisions. Under such circumstances,
the family of frequentist posteriors can be reduced to a single frequentist
posterior by the use of exact or approximate confidence sets or by
an automatic reduction rule (§\ref{sub:Collapse}). For a single frequentist
posterior, confidence decision theory is equivalent to the minimization
of expected posterior loss (§\ref{sub:Decision-degenerate}). As a
probability measure on hypothesis space, the resulting frequentist
posterior satisfies the same coherence axioms as the Bayesian posterior
whether or not it is compatible with any prior distribution (§\ref{sub:The-Bayesian-framework}).
The important special case of a scalar parameter of interest provides
an arena for contrasting frequentist posterior probabilities and \emph{p}-values
(§\ref{sub:Scalar-parameter}).

The confidence framework provides direct and simple approaches to
common problems of data analysis, as will be illustrated by example
in Sections \ref{sub:Decision-degenerate} and \ref{sub:Confidence-vs-p}.
Examples include reporting probabilistic levels of confidence of the
interval, two-sided null hypotheses required in bioequivalance testing,
assigning confidence to a complex region, and assessing practical
or scientific significance. Posterior point estimates and predictions
that account for parameter uncertainty are also available without
relinquishing the objectivity of the Neyman-Pearson framework.

Section \ref{sec:Discussion} concludes the paper by highlighting
the main properties of the proposed framework.

\section{\label{sec:General-framework}Confidence decision theory}

\subsection{\label{sub:Preliminaries}Preliminaries}

\subsubsection{Basic notation}

The values of $x\wedge y$ and $x\vee y$ are respectively the minimum
and maximum of $x$ and $y.$ The symbols $\subseteq$ and $\subset$
respectively signify subset and proper subset. $1_{\Theta^{\prime}}:\Theta\rightarrow\left\{ 0,1\right\} $
is the usual indicator function: $1_{\Theta^{\prime}}\left(\theta\right)$
is 1 if $\theta\in\Theta^{\prime}$ or 0 if $\theta\notin\Theta^{\prime}$.

Angular brackets rather than parentheses signal numeric tuples. For
example, if $x$ and $y$ are numbers, then $\left\langle x,y\right\rangle $
denotes an ordered pair, whereas $\left(x,y\right)$ denotes the open
interval $\left\{ z:x<z<y\right\} .$ 

Given a probability space $\left(\Omega,\Sigma,P_{\xi}\right)$ indexed
by the vector parameter $\xi\in\Xi\subseteq\mathbb{R}^{d},$ consider
the random quantity $X$ of distribution $P_{\xi}$ and with a realization
$x$ in some sample set $\Omega\subseteq\mathbb{R}^{n}$. Without
loss of generality, partition the full parameter $\xi$ into an interest
parameter $\theta\in\Theta$ and, unless $\theta=\xi,$ a nuisance
parameter $\gamma\in\Gamma,$ such that $\xi\in\Theta\times\Gamma$
and $P_{\theta,\gamma}=P_{\xi}.$

Except where otherwise noted, every probability distribution is a
standard (Kolmogorov) probability measure. An \emph{incomplete }probability
measure is a standard, additive measure with total mass less than
or equal to 1.

Let $\left(\Theta,\mathcal{A}\right)$ represent a measurable space
and $\mathcal{B}\left(\left[0,1\right]\right)$ the Borel $\sigma$-field
of $\left[0,1\right]$. The complement and power set of $\Theta^{\prime}$
are $\bar{\Theta^{\prime}}$ and $2^{\Theta^{\prime}},$ respectively.
The $\sigma$-field induced by $\mathcal{C}$ is $\sigma\left(\mathcal{C}\right)$.

\subsubsection{Metameasure and metaprobability spaces}

The following slight extension of probability theory is facilitates
a clear and precise presentation of the present framework. To avoid
unnecessary confusion between single-valued probability and the specific
type of multi-valued probability required, the former will be called
{}``probability'' in agreement with common usage, and the latter
will be called {}``metaprobability,'' a term defined below. 
\begin{defn}
Given a measurable space $\left(\Theta,\mathcal{A}\right)$ and a
\emph{metameasure space}, the triple $\mathcal{M}=\left(\Theta,\mathcal{A},\mathfrak{P}\right)$
with a family $\mathfrak{P}$ of measures, the \emph{metameasure }$\mathcal{P}$
of $\mathcal{M}$ is a function $\mathcal{P}$ from $\mathcal{A}$
to the set of all closed intervals of $\left[0,\infty\right)$ such
that $\mathcal{P}\left(A\right)$ is the closure of \[
\left\{ P\left(A\right):P\in\mathfrak{P}\right\} \]
for each $A\in\mathcal{A}.$ The metameasure $\mathcal{P}$ is said
to be \emph{degenerate} if $\left|\mathcal{\mathfrak{P}}\right|=1$
or \emph{nondegenerate} if $\left|\mathcal{\mathfrak{P}}\right|>1.$
\end{defn}

\begin{defn}
\label{def:expectation-interval}The metameasure $\mathcal{P}$ of
a metameasure space $\mathcal{M}=\left(\Theta,\mathcal{A},\mathfrak{P}\right)$
is a \emph{probability metameasure} if each member of $\mathfrak{P}$
is a probability measure. Then $\mathcal{M}$ is called a \emph{metaprobability
space}, and $\mathcal{P}\left(A\right)$ is called the \emph{metaprobability}
of \emph{event} $A$ for all $A\in\mathcal{A}.$ The \emph{expectation
interval} or \emph{expected interval} $\mathsf{\mathfrak{E}}\left(L\right)$
of a measurable map $L:\mathcal{A}\rightarrow\mathbb{R}^{1}$ with
respect to a probability metameasure $\mathcal{P}$ on $\mathcal{M}$
is the closure of \[
\left\{ \int L\left(\vartheta\right)dP\left(\vartheta\right):P\in\mathfrak{P}\right\} .\]

\end{defn}
In words, the expectation interval of a random quantity with respect
to a probability metameasure is the smallest closed interval containing
the expectation values of the random quantity with respect to the
probability measures of the metaprobability space.

\subsection{\label{sub:Confidence-measures-etc}Confidence measures and metameasures}

Particular types of confidence sets form the basis of the metameasure
on which confidence decision theory rests.
\begin{defn}
A \emph{set estimator} $\hat{\Theta}$ for $\theta$ is a function
defined on $\Omega\times\left[0,1\right].$ A set estimator is called
\emph{valid} if its coverage rate over repeated sampling is at least
as great as $\rho,$ the nominal confidence coefficient: \[
P_{\xi}\left(\theta\in\hat{\Theta}\left(X;\rho\right)\right)\ge\rho\]
for all $\xi\in\Xi$ and $\rho\in\left[0,1\right].$ A set estimator
is called \emph{nonconservative} if its coverage rate over repeated
sampling is at no greater than the nominal confidence coefficient:
\[
P_{\xi}\left(\theta\in\hat{\Theta}\left(X;\rho\right)\right)\le\rho\]
for all $\xi\in\Xi$ and $\rho\in\left[0,1\right].$ A set estimator
that is both valid\emph{ }and nonconservative\emph{ }is called \emph{exact}.
For some set $\mathcal{C}$ of connected subsets of $\mathcal{C},$
a set estimator is called \emph{nested }if it is a function $\hat{\Theta}:\Omega\times\left[0,1\right]\rightarrow\mathcal{C}$
such that such that, for all $x\in\Omega,$ there is a $\mathcal{C}\left(x\right)\subseteq\mathcal{C}$
such that $\hat{\Theta}\left(x;\bullet\right):\left[0,1\right]\rightarrow\mathcal{C}\left(x\right)$
is bijective, $\hat{\Theta}\left(x;0\right)=\emptyset,$ $\hat{\Theta}\left(x;1\right)=\Theta,$
and \begin{equation}
\hat{\Theta}\left(x;\rho_{1}\right)\subseteq\hat{\Theta}\left(x;\rho_{2}\right)\label{eq:nesting}\end{equation}
for all $0\le\rho_{1}\le\rho_{2}\le1.$ Two nested set estimators
$\hat{\Theta}_{1}:\Omega\times\left[0,1\right]\rightarrow\mathcal{C}$
and $\hat{\Theta}_{2}:\Omega\times\left[0,1\right]\rightarrow\mathcal{C}$
are \emph{dual} if the ranges $\mathcal{C}_{1}\left(x\right)$ and
$\mathcal{C}_{2}\left(x\right)$ of $\hat{\Theta}_{1}\left(x;\bullet\right)$
and $\hat{\Theta}_{2}\left(x;\bullet\right)$ induce the same $\sigma$-field,
i.e., $\sigma\left(\mathcal{C}_{1}\left(x\right)\right)=\sigma\left(\mathcal{C}_{2}\left(x\right)\right),$
for each $x\in\Omega.$
\end{defn}
The desired metameasure will be constructed from two confidence measures
in turn constructed from dual nested set estimators.
\begin{defn}
Let $\hat{\Theta}:\Omega\times\left[0,1\right]\rightarrow\mathcal{C}$
denote a nested set estimator and $\mathcal{A}^{x}$ the $\sigma$-field
induced by $\mathcal{C}\left(x\right),$ the range of $\hat{\Theta}\left(x;\bullet\right)$
for each $x\in\Omega.$ Then, for all $x\in\Omega,$ \emph{$\hat{\Theta}$
induces} the probability space $\left(\Theta,\mathcal{A}^{x},P^{x}\right)$
and the \emph{confidence measure} or \emph{frequentist posterior}
$P^{x},$ the probability measure on $\mathcal{A}^{x}$ such that\begin{equation}
\Theta^{\prime}\in\mathcal{C}\left(x\right)\implies\Theta^{\prime}=\hat{\Theta}\left(x;P^{x}\left(\Theta^{\prime}\right)\right).\label{eq:certainty-measure}\end{equation}
The probability $P^{x}\left(\Theta^{\prime}\right)$ is the \emph{confidence
level} of the hypothesis that $\theta\in\Theta^{\prime}.$ If $\hat{\Theta}$
is valid, nonconservative, or exact, then $P^{x}$ and $P^{x}\left(\Theta^{\prime}\right)$
are likewise called \emph{valid}, \emph{nonconservative}, or \emph{exact},
respectively.
\end{defn}
The next result provides the confidence level of any hypothesis that
$\theta\in\Theta^{\prime}\in\mathcal{A}^{x}$ as the sum of confidence
levels given more directly by equation (\ref{eq:certainty-measure}).
\begin{prop}
For each $x\in\Omega,$ let $\left(\Theta,\mathcal{A}^{x},P^{x}\right)$
be the confidence measure induced by the nested set estimator $\hat{\Theta}:\Omega\times\left[0,1\right]\rightarrow\mathcal{C},$
and let $\mathcal{C}\left(x\right)$ be the range of $\hat{\Theta}\left(x;\bullet\right).$
For some $K\in\left\{ 1,2,\dots\right\} ,$ let $\Theta^{\prime}=\cup_{k=1}^{K}\Theta_{k}^{\prime},$
where $\Theta_{k}^{\prime}\in\mathcal{A}^{x}$ and \[
i\ne j\implies\Theta_{i}^{\prime}\cap\Theta_{j}^{\prime}=\emptyset.\]
Then\begin{equation}
P^{x}\left(\Theta^{\prime}\right)=\sum_{k=1}^{K}\left(P^{x}\left(\Theta_{k}^{+}\right)-P^{x}\left(\Theta_{k}^{-}\right)\right),\label{eq:sum-of-interval-P}\end{equation}
where\[
\Theta_{k}^{+}=\arg\inf_{\Theta^{\prime\prime}\in\mathcal{C}\left(x\right),\Theta^{\prime}\subseteq\Theta^{\prime\prime}}\left|\Theta^{\prime\prime}\right|\]
and $\Theta_{k}^{-}=\Theta_{k}^{+}\backslash\Theta_{k}^{\prime}$
for all $k\in\left\{ 1,2,\dots,K\right\} .$~\end{prop}
\begin{proof}
$P^{x}\left(\Theta_{k}^{+}\right)=P^{x}\left(\Theta_{k}^{\prime}\right)+P^{x}\left(\Theta_{k}^{-}\right)$
and $P^{x}\left(\cup_{k=1}^{K}\Theta_{k}^{\prime}\right)=\sum_{k=1}^{K}P^{x}\left(\Theta_{k}^{\prime}\right)$
follow from the mutual exclusivity of the sets and from the additivity
of the measure $P^{x}.$ 
\end{proof}
Thus, since, for all $k\in\left\{ 1,2,\dots,K\right\} ,$ both $\Theta_{k}^{+}$
and $\Theta_{k}^{-}$ are in $\mathcal{C}\left(x\right)$ and since
$\mathcal{C}\left(x\right)$ induces $\mathcal{A}^{x},$ equations
(\ref{eq:certainty-measure}) and (\ref{eq:sum-of-interval-P}) can
be used to calculate $P^{x}\left(\Theta^{\prime}\right)$ for any
$\Theta^{\prime}\in\mathcal{A}^{x}.$
\begin{defn}
\label{def:confidence-metameasure}Consider the dual nested set estimators
$\Theta_{\ge}:\Omega\times\left[0,1\right]\rightarrow\mathcal{C},$
which is valid, and\emph{ }$\Theta_{\le}:\Omega\times\left[0,1\right]\rightarrow\mathcal{C},$
which is nonconservative. For every $x\in\Omega,$ let $\mathcal{A}^{x}$
denote the common $\sigma$-field induced by each of the ranges of
$\hat{\Theta}_{\ge}\left(x;\bullet\right)$ and $\hat{\Theta}_{\le}\left(x;\bullet\right).$
If $P_{\ge}^{x}$ is the \emph{valid confidence measure}, the confidence
measure induced by $\Theta_{\ge},$ then $P_{\ge}^{x}\left(\Theta^{\prime}\right)$
is called a \emph{valid confidence level} of the hypothesis that $\theta\in\Theta^{\prime}.$
For each $x\in\Omega,$ the dual \emph{nonconservative confidence
measure} $P_{\le}^{x}$ and \emph{nonconservative confidence level}
$P_{\le}^{x}\left(\Theta^{\prime}\right)$ are defined analogously.
On the metaprobability space \[
\mathcal{M}_{\ge,\le}^{x}=\left(\Theta,\mathcal{A}^{x},\left\{ P_{\ge}^{x},P_{\le}^{x}\right\} \right),\]
called a \emph{confidence metameasure space}, the probability metameasure
$\mathcal{P}^{x}$ is called the \emph{confidence metameasure} \emph{induced
by }$\hat{\Theta}_{\ge}$ \emph{and }$\hat{\Theta}_{\le}$ \emph{given
some $x$ in }$\Omega.$ Accordingly, the \emph{confidence metalevel
}of the hypothesis that $\theta\in\Theta^{\prime}$ is $\mathcal{P}^{x}\left(\Theta^{\prime}\right)$
for all $\Theta^{\prime}\in\mathcal{A}^{x}.$ By the definition of
metaprobability, any hypothesis $\Theta^{\prime}\in\mathcal{A}^{x}$
has a confidence metalevel of \begin{equation}
\mathcal{P}^{x}\left(\Theta^{\prime}\right)=\left[P_{\ge}^{x}\left(\Theta^{\prime}\right)\wedge P_{\le}^{x}\left(\Theta^{\prime}\right),P_{\ge}^{x}\left(\Theta^{\prime}\right)\vee P_{\le}^{x}\left(\Theta^{\prime}\right)\right].\label{eq:metalevel}\end{equation}
\end{defn}
\begin{rem}
The restriction to $\sigma$-fields with events common to valid and
nonconservative confidence measures strongly constrains the choice
of the estimators to ensure the ability to assign a confidence metalevel
to any hypothesis of interest without a need for incomplete probability
measures. The further flexibility of allowing multiple $\sigma$-fields
in a class of measure spaces may be desirable in some applications. 
\end{rem}

\label{para:Strategies}Strategies developed within more conventional
frequentist frameworks provide guidance on the choice of which dual
set estimators by which to induce the confidence metameasure. Extending
the statistical model to incorporate information from the physics
of experimental design and measurement can rule out many pathological
set estimators as meaningless \citep{RefWorks:1310}. For instance,
the inclusion of transformation-group structure in the model leads
to set estimators that exactly match Bayesian posterior credible sets
under certain improper priors \citep{FRASER1968b,RefWorks:1443}.
Without taking advantage of extended models, \citet[121-122, 132-133]{RefWorks:436},
\citet[pp. 75-76]{RefWorks:SprottBook2000}, and \citet{RefWorks:437}
highlight advantages of incorporating information from the likelihood
function into set estimators; cf. Section \ref{sub:Likelihood}. 
\begin{example}[normal distribution]
\label{exa:normal}For $n$ independent random variables each distributed
according to $P_{\theta,\gamma},$ the normal distribution with mean
$\theta$ and variance $\gamma,$ the interval estimator $\Theta^{\alpha}$
given by \[
\Theta^{\alpha}\left(x;\rho\right)=\left[p_{x}^{-1}\left(\alpha\right),p_{x}^{-1}\left(\rho+\alpha\right)\right]\]
for all $\rho\in\left[0,1-\alpha\right]$ is nested and is an exact
$\rho\left(100\%\right)$ confidence interval for $\theta,$ where
$\alpha\in\left[0,1\right],$ $p_{x}\left(\theta^{\prime}\right)$
is the upper-tailed \emph{p}-value of the hypothesis that $\theta=\theta^{\prime},$
and $p_{x}^{-1}$ is the inverse of $p_{x}.$  Since $\Theta^{\alpha}$
is both valid and nonconservative, it is dual to itself, yielding
the equality of the valid and nonconservative confidence measures
$P_{\alpha,\ge}^{x}$ and $P_{\alpha,\le}^{x},$ each the distribution
of \[
\vartheta=\bar{x}+T_{n-1}\hat{\sigma}/\sqrt{n},\]
where $T_{n-1}$ is the random variable of the Student \emph{t} distribution
with $n-1$ degrees of freedom. Hence, the confidence metameasure
$\mathcal{P}_{\alpha}^{x}$ induced by $\Theta^{\alpha}$ is degenerate:\[
\left(\Theta,\mathcal{A}^{x},\left\{ P_{\alpha,\ge}^{x},P_{\alpha,\le}^{x}\right\} \right)=\left(\Theta,\mathcal{A}^{x},\left\{ P_{\alpha}^{x}\right\} \right)\]
If $\Theta^{\prime}$ is an interval, then\[
P_{\alpha}^{x}\left(\Theta^{\prime}\right)=p_{x}\left(\sup\Theta^{\prime}\right)-p_{x}\left(\inf\Theta^{\prime}\right)\]
for all $x\in\Omega$ and $\Theta^{\prime}\in\mathcal{A}^{x},$ from
which it follows that the confidence measure $P_{\alpha}^{x}$ does
not depend on the nested set estimator chosen and can thus be represented
by $P^{x}.$ 
\end{example}
Special properties of degenerate confidence metameasures are given
in Section \ref{sec:Single-measure-approximation}. The next example
involves a nondegenerate confidence metameasure.
\begin{example}[binomial distribution]
\label{exa:binomial}Let $P_{\theta}$ denote the binomial measure
with $n$ trials, success probability $\theta\in\Theta,$ and $C$-corrected,
upper-tail cumulative probabilities $p_{C,x}\left(\theta\right)=P_{\theta}\left(X>x\right)+CP_{\theta}\left(X=x\right),$
with $C\in\left[0,1\right].$ Consider the family $\mathcal{F}_{C}=\left\{ \Theta_{C}^{\alpha}:\alpha\in\left[0,1\right]\right\} $
of nested set estimators such that\[
\Theta_{C}^{\alpha}\left(x;\rho\right)=\begin{cases}
\left[p_{1-C,x}^{-1}\left(\alpha\right),p_{C,x}^{-1}\left(\alpha+\rho\right)\right] & \rho\in\left(0,1-\alpha\right]\\
\emptyset & \rho=0\\
\left[0,1\right] & \rho=1\end{cases}\]
for all $\alpha\in\left[0,1\right],$ $\rho\in\mathfrak{R}=\left[0,1-\alpha\right]\cup\left\{ 1\right\} ,$
$x\in\left\{ 0,1,...\right\} =\Omega,$ where\begin{equation}
p_{C^{\prime},x}^{-1}\left(\alpha^{\prime}\right)=\theta^{\prime}\iff p_{C^{\prime},x}\left(\theta^{\prime}\right)=\alpha^{\prime}.\label{eq:binomial}\end{equation}
Since the rates at which the valid $\left(C=0\right)$ and nonconservative
$\left(C=1\right)$ interval estimators cover $\theta$ are bound
according to \[
P_{\theta}\left(\theta\in\Theta_{0}^{\alpha}\left(X;\rho\right)\right)\ge\rho,\]
\[
P_{\theta}\left(\theta\in\Theta_{1}^{\alpha}\left(X;\rho\right)\right)\le\rho,\]
the sets $\mathcal{F}_{0}$ and $\mathcal{F}_{1}$ are valid and nonconservative
families of nested set estimators, respectively, and for any $\alpha\in\left[0,1\right],$
the valid set estimator $\Theta_{0}^{\alpha}$ is dual to the nonconservative
set estimator $\Theta_{1}^{\alpha},$ thus inducing the valid confidence
measure $P_{\alpha,0}^{x},$ the nonconservative confidence measure
$P_{\alpha,1}^{x},$ and the confidence metameasure $\mathcal{P}_{\alpha}^{x}$
on the $\sigma$-field $\mathcal{B}\left(\left[0,1\right]\right)$
for each $x\in\Omega.$ In order to weigh evidence in $X=x$ for the
hypothesis that $0\le\theta^{\prime}\le\theta\le\theta^{\prime\prime}\le1,$
equation (\ref{eq:certainty-measure}) furnishes \[
P_{\alpha,C}^{x}\left(\left[p_{1-C,x}^{-1}\left(\alpha\right),p_{C,x}^{-1}\left(\alpha+\rho_{C,x}\right)\right]\right)=\rho_{C,x},\]
and, with equation (\ref{eq:sum-of-interval-P}), \begin{eqnarray}
P_{\alpha,C}^{x}\left(\left[\theta^{\prime},\theta^{\prime\prime}\right]\right) & = & P_{\alpha,C}^{x}\left(\left[p_{1-C,x}^{-1}\left(\alpha\right),p_{C,x}^{-1}\left(\alpha+\rho_{C,x}^{\prime\prime}\right)\right]\right)-P_{\alpha,C}^{x}\left(\left[p_{1-C,x}^{-1}\left(\alpha\right),p_{C,x}^{-1}\left(\alpha+\rho_{C,x}^{\prime}\right)\right]\right)\label{eq:binomial-confidence-level}\\
 & = & \rho_{C,x}^{\prime\prime}-\rho_{C,x}^{\prime},\nonumber \end{eqnarray}
where\begin{eqnarray*}
\rho_{C,x}^{\prime} & = & p_{C,x}\left(\theta^{\prime}\right)-\alpha\\
\rho_{C,x}^{\prime\prime} & = & p_{C,x}\left(\theta^{\prime\prime}\right)-\alpha.\end{eqnarray*}
Since $\alpha$ drops out of the difference, let $P_{C}^{x}=P_{\alpha,C}^{x}.$
For any $\Theta^{\prime}\in\mathcal{B}\left(\left[0,1\right]\right),$
equations (\ref{eq:binomial-confidence-level}) and (\ref{eq:metalevel})
specify the confidence metalevel of the hypothesis that $\theta\in\Theta^{\prime}.$
To illustrate the reduction of confidence indeterminacy with additional
observations, the boundary values of $\mathcal{P}^{x}\left(\left[1/4,3/4\right]\right)$
are plotted against $n$ in Fig. \ref{fig:binomial} for the $\theta=2/3$
case. 
\end{example}
~%
\begin{figure}
\includegraphics[scale=0.35]{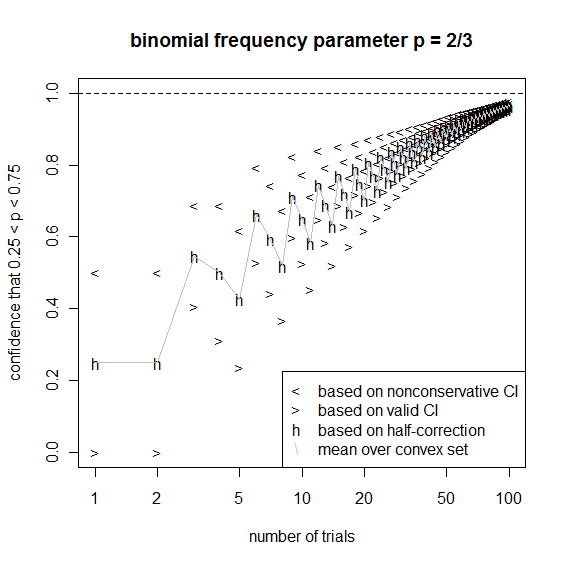}

\caption{Confidence levels of the hypothesis that $\theta,$ the limiting relative
frequency of successes, is between 1/4 and 3/4 as a function of $n,$
the number of independent trials, with $\theta=2/3$ as the unknown
true value. In the notation of Example \ref{exa:binomial}, the nonconservative
confidence level is $P_{1}^{x}\left(\left[1/4,3/4\right]\right),$
the valid confidence level is $P_{0}^{x}\left(\left[1/4,3/4\right]\right),$
and the half-corrected level is $P_{1/2}^{x}\left(\left[1/4,3/4\right]\right).$
The confidence level averaged over the convex set is defined in Section
\ref{sub:Collapse}. Sampling variation was suppressed by setting
each number $x$ of successes to the lowest integer greater than or
equal to $n\theta$ instead of randomly drawing values of $x$ from
the $\left\langle n,\theta\right\rangle $ binomial distribution.
\label{fig:binomial}}

\end{figure}

\subsection{\label{sub:Coherence}Coherence of confidence metalevels}

The confidence metameasure $\mathcal{P}^{x}$ on confidence space
$\mathcal{M}_{\ge,\le}^{x}$ models the reasoning process of an ideal
agent betting on inclusion of the true parameter value in elements
of $\mathcal{A}^{x},$ the $\sigma$-field of $\mathcal{M}_{\ge,\le}^{x},$
with upper and lower betting odds determined by the coverage rates
of the corresponding valid and nonconservative confidence sets. The
coherence of the agent's decisions may be evaluated by expressing
its betting odds in terms of upper and lower probabilities that lack
the additivity property of Kolmogorov's probability measures. Given
the dual functions $u:\mathcal{A}^{x}\rightarrow\left[0,1\right]$
and $v:\mathcal{A}^{x}\rightarrow\left[0,1\right]$ such that\begin{equation}
u\left(\Theta^{\prime}\right)+v\left(\Theta\backslash\Theta^{\prime}\right)=1,\label{eq:dual}\end{equation}
\[
u\left(\Theta^{\prime}\cup\Theta^{\prime\prime}\right)\ge u\left(\Theta^{\prime}\right)+u\left(\Theta^{\prime\prime}\right),\]
\[
v\left(\Theta^{\prime}\cup\Theta^{\prime\prime}\right)\le v\left(\Theta^{\prime}\right)+v\left(\Theta^{\prime\prime}\right)\]
for all disjoint $\Theta^{\prime}$ and $\Theta^{\prime\prime}$ in
$\mathcal{A}^{x},$ the values $u\left(\Theta^{\prime}\right)$ and
$v\left(\Theta^{\prime}\right)$ are the \emph{lower }and \emph{upper
probabilities} \citep[§9.3]{Molchanov2005b} of the hypothesis that
$\theta\in\Theta^{\prime}.$ The decision-theoretic interpretation
is that $u\left(\Theta^{\prime}\right)$ is the largest price an agent
would pay for a gain of $1_{\theta}\left(\Theta^{\prime}\right),$
whereas $v\left(\Theta^{\prime}\right)$ is the smallest price for
which the same agent would sell that gain, assuming an additive utility
function \citep{RefWorks:Walley1991}. The duality between $u$ and
$v$ expressed as equation (\ref{eq:dual}) means each function is
completely determined by the other. 

The function $u$ is called the \emph{lower envelope} of a family
$\mathfrak{P}$ of measures on $\mathcal{A}$ if

\[
u\left(\Theta^{\prime}\right)=\inf_{P\in\mathfrak{P}}P\left(\Theta^{\prime}\right)\]
for all $\Theta^{\prime}\in\mathcal{A}$ (\citealp[§15.2]{ColettiScozzafava2002b};
\citet[§9.3]{Molchanov2005b}). Since the lower envelope\emph{ }of
a family of probability measures is a \emph{coherent lower probability
}(\citealp[§3.3.3]{RefWorks:Walley1991}; \citet[§9.3]{Molchanov2005b})
and since $\left\{ P_{\ge}^{x},P_{\le}^{x}\right\} $ as specified
in Definition \ref{def:confidence-metameasure} constitutes such a
family, the agent weighing evidence for any hypothesis $\theta\in\Theta^{\prime}$
by $\mathcal{P}^{x}\left(\Theta^{\prime}\right),$ with $\Theta^{\prime}\in\mathcal{A}^{x},$
satisfies the minimal set of rationality axioms of \citet{RefWorks:Walley1991}.
It follows that the agent avoids sure loss by making decisions according
to the lower and upper probabilities \[
u\left(\Theta^{\prime}\right)=P_{\ge}^{x}\left(\Theta^{\prime}\right)\wedge P_{\le}^{x}\left(\Theta^{\prime}\right),\]
\[
v\left(\Theta^{\prime}\right)=1-u\left(\Theta\backslash\Theta^{\prime}\right).\]

Conversely, the framework of Section \ref{sub:Confidence-measures-etc}
can be presented starting with de Finetti's prevision and the related
concept of coherent extension \citep{RefWorks:Walley1991,ColettiScozzafava2002b}
as follows. An intelligent agent first sets its prices for buying
and selling gambles on the hypotheses corresponding to the elements
of $\mathcal{C}$ according to the confidence coefficients of valid
and nonconservative nested set estimators. Then it extends its prices
or \emph{previsions }to the family of the two probability measures
on the $\sigma$-field induced by $\mathcal{C}$ in order to evaluate
the probability of a hypothesis $\theta\in\Theta^{\prime}$ for some
$\Theta^{\prime}$ in the $\sigma$-field but not in $\mathcal{C}.$
This family in turn yields coherent lower and upper probabilities
that equal the initial buying and selling prices whenever the latter
apply, i.e., when the hypothesis is that $\theta\in\Theta^{\prime}$
for some $\Theta^{\prime}\in\mathcal{C}.$ Thus, a Dutch book cannot
be made against the agent.

\subsection{\label{sub:Decisions-given-confidence}Decisions under arbitrary
loss}

This section generalizes betting under 0-1 loss to making confidence-based
decisions under any unbounded loss function. Confidence metalevels
do not describe the actual betting behavior of any human agent, but
instead prescribe decisions, including amounts bet on any hypothesis
involving $\theta,$ given that the agent will incur a loss of $L_{a}\left(\theta\right)$
for taking action $a.$ 

According to a natural generalization of the Bayes decision rule of
minimizing loss averaged over a posterior distribution, action $a^{\prime}$
\emph{dominates} (is rationally preferred to) action $a^{\prime\prime}$
if and only if \[
\forall\mathcal{E}^{\prime}\in\mathsf{\mathfrak{E}}\left(L_{a^{\prime}}\right),\mathcal{E}^{\prime\prime}\in\mathsf{\mathfrak{E}}\left(L_{a^{\prime\prime}}\right):\mathcal{E}^{\prime}\le\mathcal{E}^{\prime\prime}\]
\[
\exists\mathcal{E}^{\prime}\in\mathsf{\mathfrak{E}}\left(L_{a^{\prime}}\right),\mathcal{E}^{\prime\prime}\in\mathsf{\mathfrak{E}}\left(L_{a^{\prime\prime}}\right):\mathcal{E}^{\prime}<\mathcal{E}^{\prime\prime},\]
where both expectation intervals (Definition \ref{def:expectation-interval})
are with respect to the same confidence metameasure $\mathcal{P}^{x}.$
The confidence metameasures impose no restrictions on agent decisions
other than restricting them to non-dominated actions. 

This use of the confidence metameasure in making decisions follows
a previous generalization of maximizing expected utility to multi-valued
probability. (Here, the utilities are expressed in terms of equivalent
losses, as is conventional in the statistics literature.) Kyburg (\citeyear[pp. 180, 231-234]{RefWorks:1486};
\citeyear{KyburgJrDegrees2003139,Kyburg2006}) and \citet[§1.4]{Kaplan1996b}
used the principle of dominance to make decisions on the basis of
intervals of expected utilities determined by the expected utility
of each probability measure: an action yielding expected utilities
in interval $A$ is preferred to that yielding expected utilities
in interval $B$ if at least one member of $A$ is greater than all
members of $B$ and if no member of $A$ is less than any member of
$B$. 

While multi-valued probabilities do not dictate how to choose one
of the non-dominated actions in situations that demand a choice equivalent
to deciding between accepting a hypothesis or accepting its alternative,
they may prove more practical when indecision can be broken by additional
considerations, as \citet[pp. 161-162, 235-241]{RefWorks:Walley1991}
explained. In the case of a human agent, \citet{KyburgJrDegrees2003139}
argued for selecting among non-dominated actions on the basis of considerations
that cannot be represented mathematically rather than selecting on
the basis of an arbitrary prior distribution. 

If a single-valued estimate of $1_{\Theta^{\prime}}\left(\theta\right)$
is needed for some $\Theta^{\prime}\in\mathcal{A},$ the \emph{indeterminacy}
$\sup\mathcal{P}^{x}\left(\Theta^{\prime}\right)-\inf\mathcal{P}^{x}\left(\Theta^{\prime}\right)$
can quantify a set estimator's degree of undesirable conservatism;
some ways to eliminate such indeterminacy by replacing a confidence
metameasure with a confidence measure are mentioned in Section \ref{sec:Single-measure-approximation}.
If indeterminacy is removed, the above dominance principle reduces
to the principle of minimizing expected loss (§\ref{sub:Decision-degenerate}).

\subsection{\label{sub:Likelihood}Likelihood principle}

While in some cases the likelihood function can guide the construction
of set estimators with desirable properties, as noted in Section \ref{para:Strategies},
it plays no general role in confidence decision theory. Consequently,
inference does not always obey the likelihood principle: some set
estimators lead to values of evidential support and partial proof
that depend on information in the sampling model not encoded in the
likelihood function; cf. \citet{Wilkinson1977}. 

An advantage of coherent statistical methods in general is the flexibility
they give the researcher to simultaneously consider as many hypotheses
and interval estimates for $\theta$ as desired. Although such versatility
is usually presented as a consequence of the likelihood principle
and Bayesian statistics, they are not needed to secure it once coherence
has been established (§\ref{sub:Coherence}). 

That the proposed framework is not constrained by the likelihood principle
distinguishes it from Peter Walley's $W_{1}$ and $W_{2}$, two inferential
theories of indeterminate (multi-valued) probability intended to satisfy
the best aspects of both coherence and frequentism \citep{RefWorks:1225}.
The coverage error rate of $W_{1}$ tends to be much higher than the
nominal rate in order to ensure simultaneous compliance with the likelihood
principle. Although the principle often precludes approximately correct
frequentist coverage, more power can be achieved by less stringently
controlling the error rate \citep{RefWorks:1225}. \citet{RefWorks:1225}
did not report the degree of conservatism of $W_{2}$, a normalized
likelihood method. With a uniform measure for integration over parameter
space, the normalized likelihood is equal to the Bayesian posterior
that results from a uniform prior.

\section{\label{sec:Single-measure-approximation}Frequentist posterior distribution}

An important realm for practical applications of the above framework
is the situation in which inference may reasonably depend only on
a single confidence measure $P^{x}$ rather than directly on a confidence
metameasure $\mathcal{P}^{x}.$ That is possible not only in the special
case of degeneracy due to the availability of a suitable exact nested
set estimator (Example \ref{exa:normal}), but can also be achieved
either by transforming a nondegenerate confidence metameasure to a
confidence measure (§\ref{sub:Collapse}) or by approximating a confidence
measure. Remark \ref{rem:asymptotics} concerns the latter strategy
in the case of a scalar parameter of interest. 

Relying solely on a single confidence measure for inference and decision
making (§\ref{sub:Decision-degenerate}) enjoys the coherence of theories
of utility maximization usually associated with Bayesianism (§\ref{sub:The-Bayesian-framework}).
In the ubiquitous special case of a scalar parameter of interest,
a single confidence level of a hypothesis is a consistent estimator
of whether the hypothesis is true under more general conditions than
is the \emph{p}-value as such an estimator (§\ref{sub:Scalar-parameter}).

\subsection{\label{sub:Collapse}Reducing a confidence metameasure }

Interpreting upper and lower probabilities as bounds defining a family
of permissible probability measures, \citet{Williamson2007b} argued
for minimizing expected loss with respect to a single distribution
within the family instead of using outside considerations to choose
among actions that are non-dominated in the sense of Section \ref{sub:Decisions-given-confidence}.
Consider the confidence metameasure space $\mathcal{M}_{\ge,\le}^{x}=\left(\Theta,\mathcal{A}^{x},\left\{ P_{\ge}^{x},P_{\le}^{x}\right\} \right)$
of confidence metameasure $\mathcal{P}^{x}$ for some $x\in\Omega.$
A much larger family $\mathfrak{P}$ of measures on $\mathcal{A}^{x}$
such that $\left\{ P_{\ge}^{x},P_{\le}^{x}\right\} $ and $\mathfrak{P}$
have the same lower envelope $u$ is the convex set \[
\mathfrak{P}=\left\{ P_{D}^{x}:D\in\left[0,1\right]\right\} ,\]
where $P_{D}^{x}=\left(1-D\right)P_{\ge}^{x}+DP_{\le}^{x},$ thereby
forming the metaprobability space $\tilde{\mathcal{M}}_{\ge,\le}^{x}=\left(\Theta,\mathcal{A}^{x},\mathfrak{P}\right)$
and probability metameasure $\tilde{\mathcal{P}}^{x};$ cf. \citet[§11]{RefWorks:1456};
\citet{WassermanPriorEnvelopes1990}; \citet[pp. 40-42]{Paris232410}.
Since $\tilde{\mathcal{P}}^{x}=\mathcal{P}^{x},$ the measure $P^{x}\in\mathfrak{P}$
selected according to some rule is called a \emph{reduction }of $\mathcal{P}^{x}.$ 

Effective reduction of $\mathcal{P}^{x}$ to a single measure $P^{x}$
can be accomplished by averaging over $\mathfrak{P}$ with respect
to the Lebesgue measure. That average of the convex set is simply
the mean of the valid and nonconservative confidence measures:\begin{equation}
P^{x}\left(\Theta^{\prime}\right)=\int_{0}^{1}P_{D}^{x}\left(\Theta^{\prime}\right)dD=\left(P_{\ge}^{x}\left(\Theta^{\prime}\right)+P_{\le}^{x}\left(\Theta^{\prime}\right)\right)/2=P_{1/2}^{x}\left(\Theta^{\prime}\right)\label{eq:mean-of-convex-set}\end{equation}
for all $\Theta^{\prime}\in\mathcal{A}^{x};$ recall that $P_{1/2}^{x}\in\mathfrak{P}.$

Other automatic methods of reducing a metameasure to a single measure
are also available. For example, the recommendation of \citet{Williamson2007b}
to select the measure within the family that maximizes the entropy
is minimax under Kullback-Leibler loss \citep{Gruenwald20041367}.
\begin{example}[Binomial distribution, continued from Example \ref{exa:binomial}]
 As the gray line in Fig. \ref{fig:binomial} indicates, the mean
measure $P^{x}$ of the convex set (\ref{eq:mean-of-convex-set})
yields a confidence level between those of the valid and nonconservative
confidence measures, discarding the notable reduction in confidence
nondegeneracy from $n=1$ to $n=10$ as irrelevant for action in situations
that do not permit indecision. The approximate (half-corrected) confidence
level also disregards nondegeneracy information, yielding in this
special case the same levels of confidence as does $P^{x}.$ In contrast,
the confidence metameasure records the nondegeneracy as the difference
between the agent's selling and buying prices of a gamble with a payoff
contingent on whether or not $\theta\in\left[1/4,3/4\right],$ a difference
that becomes less important as $n$ increases. 
\end{example}

\subsection{\label{sub:Decision-degenerate}Confidence-based decision and inference}

\subsubsection{\label{sub:Minimizing-expected-loss}Minimizing expected loss}

In a situation requiring a decision involving the acceptance or rejection
of the hypothesis that $\theta\in\Theta^{\prime},$ that is, under
a 0-1 loss function, an agent guided by a single measure $P^{x}$
regards $P^{x}\left(\vartheta\in\Theta^{\prime}\right)/P^{x}\left(\vartheta\notin\Theta^{\prime}\right)$
as the fair betting odds and will act accordingly. The hypothesis
$\theta\in\Theta^{\prime}$ will be accepted only if the odds $P^{x}\left(\vartheta\in\Theta^{\prime}\right)/P^{x}\left(\vartheta\notin\Theta^{\prime}\right)$
are greater than the ratio of the cost that would be incurred if $\theta\notin\Theta^{\prime}$
to the benefit that would be gained if $\theta\in\Theta^{\prime}$.
Otherwise, unless the odds are exactly equal to 1, the hypothesis
$\theta\notin\Theta^{\prime}$ will be accepted. Under a more general
class of loss functions, the decision theory of Section \ref{sub:Decisions-given-confidence}
reduces to the minimization of expected loss given the degeneracy
or reduction of the confidence metameasure. Section \ref{sub:Non-Bayesian-coherence}
notes implications for axiomatic coherence.

\subsubsection{\label{sub:Hypothesis}Applications to hypothesis assessment}

As the findings of basic science are arguably valuable even if never
applied and since the ways in which any inductive inference will be
used are often unpredictable \citep[pp. 95-96, 103-106]{RefWorks:985},
$P^{x}\left(\vartheta\in\Theta^{\prime}\right)$ may be reported as
an estimate of $1_{\Theta^{\prime}}\left(\theta\right)$ for use with
currently unknown loss functions (cf. \citealp{RefWorks:1203}; \citealt{Hwang1992490}).
That inferential role is currently played in many of the sciences
by the \emph{p}-value interpreted as a measure of evidence in {}``significance
testing'' \citep{RefWorks:294}, but its notorious lack of coherence
has prevented its universal acceptance (\citealp[e.g., ][]{RefWorks:122}).
As will become clear in Section \ref{sub:Scalar-parameter}, $P^{x}\left(\vartheta\in\Theta^{\prime}\right)$
can differ markedly from the \emph{p}-value for testing $\theta\in\Theta^{\prime}$
as the null hypothesis not only in interpretation but also in numeric
value.

~
\begin{example}
\citet[§3]{RefWorks:249} consider the hypothesis that the mean $\xi$
of a $\nu$-dimensional multivariate normal distribution of an identity
covariance matrix is in an origin-centered sphere of radius $\theta^{\prime\prime}$
but outside a concentric sphere of radius $\theta^{\prime}$. Let
$\theta=||\xi||$, and let $\chi_{\nu}^{2}$ be the chi-squared cumulative
distribution function (CDF) of $\nu$ degrees of freedom. Since the
\emph{p}-value of the null hypothesis that $\theta\ge\theta^{\prime}$
is $\chi_{\nu}^{2}\left(\left(||x||/\theta^{\prime}\right)^{2}\right)$,
the confidence level of the hypothesis that $\theta^{\prime}<\theta<\theta^{\prime\prime}$
is \begin{eqnarray*}
P^{x}\left(\theta^{\prime}<\vartheta<\theta^{\prime\prime}\right)= & \chi_{\nu}^{2}\left(\left(||x||/\theta^{\prime}\right)^{2}\right)-\chi_{\nu}^{2}\left(\left(||x||/\theta^{\prime\prime}\right)^{2}\right) & ,\end{eqnarray*}
the value of which \citet[§4]{RefWorks:249} justified as an approximation
to a Bayesian posterior probability. The coherence of the confidence
measure $P^{x}$ immunizes it against the inconsistencies that \citet[§3]{RefWorks:249}
noticed among \emph{p}-values: contradictory conclusions would be
reached depending on which hypothesis was considered as the null. 
\end{example}
A practical implication of working in the confidence metameasure framework
is that since the simple bootstrap methods of \citet{RefWorks:249}
based on a scalar pivot enable close approximations to \emph{p}-value
functions \citep{Efron19933,RefWorks:127,RefWorks:130,RefWorks:1092},
they can solve related problems too complex for more rigid Neyman-Pearson
methods and yet without any need to seek matching priors for justification;
cf. \citet{Efron2003135}. Applications include assigning levels of
confidence to phylogenetic tree branches \citet{RefWorks:250}, to
observed local maxima in an estimated function \citep{RefWorks:249,Hall20042098},
and to gene network connections found on the basis of microarray data
\citep{Kamimura2003350}. \citet{Liu1997266} studied operating characteristics
of the \emph{empirical strength probability} (ESP), which in the one-dimensional
case is equal to some confidence probability $P^{x}\left(\theta^{\prime}<\vartheta<\theta^{\prime\prime}\right)$
defined with respect to a bootstrap algorithm.

See \citet{Polansky2007b} for an accessible introduction to the
general problem of {}``observed confidence levels'' of composite
hypotheses, which \citet{RefWorks:249} had dubbed the {}``problem
of regions,'' understood to include applications to ranking and selection
as well as those mentioned above. The fundamental characteristic of
this approach is not the bootstrapping technique as much as the property
that the level of confidence in any given region is equal to the coverage
rate of a corresponding confidence set. Until the ESP is seen to have
a compelling justification of its own, it may continue to be regarded
merely as a method of last resort since it is in general neither a
Bayesian posterior probability nor a Neyman-Pearson \emph{p}-value:
{}``For {[}the latter{]} reason, it seems best to use the ESP only
when more specific, direct testing methods are not available for a
particular problem'' \citep{DavisonHinkleyYoung2003b}. That the
ESP and other approximations of the confidence value are more acceptable
than \emph{p}-values as estimates of whether the parameter lies in
a given region (§\ref{sub:Consistency-of-support}) gives cause to
reconsider that judgment even apart from the coherence of the confidence
value.

\begin{example}[beyond statistical significance]
\label{exa:hierarchical}Consider the null hypothesis $\theta^{\prime}-\Delta\le\theta\le\theta^{\prime}+\Delta$,
where the non-negative scalar $\Delta$ is a minimal degree of practical
or scientific significance in a particular application. For instance,
researchers developing methods of analyzing microarray data are increasingly
calling for specification of a minimal level of biological significance
when testing null hypotheses of equivalent gene expression against
alternative hypotheses of differential gene expression \citep{RefWorks:89,RefWorks:427,RefWorks:486,RefWorks:435}.
\citet{RefWorks:21} and \citet{DavisJMcCarthy03152009} in effect
approached the problem with \emph{p}-values of composite null hypotheses,
in conflict with the confidence measure approach (Example \ref{exa:equivalence}
and Section \ref{sub:Consistency-of-support}). 
\end{example}
Section \ref{sub:Confidence-vs-p} provides additional examples that
contrast hypothesis confidence levels with hypothesis \emph{p}-values
in practical applications.

\subsubsection{\label{sub:Estimation-prediction}Other applications of minimizing
expected loss}

The framework of minimizing expected loss with respect to a confidence
measure (§\ref{sub:Minimizing-expected-loss}) not only leads to assigning
confidence levels to hypotheses but also provides methods for optimal
estimation and prediction. In addition, confidence-measure estimators
and predictors have frequentist properties only shared with Bayesian
estimators and predictors when the Bayesian posterior is a confidence
measure.

As the frequentist posterior, the confidence measure gives all the
point estimators provided by the Bayesian posterior. For example,
the frequentist posterior mean, minimizing expected squared error
loss, is $\bar{\vartheta}_{x}=\int_{\Theta}\vartheta dP^{x}\left(\vartheta\right)$
and the frequentist posterior $p$-quantile, minimizing expected loss
for a threshold-based function of $p$ \citep[App. B]{CarlinLouis3},
is $\vartheta\left(p\right)$ such that $p=P^{x}\left(\vartheta<\vartheta\left(p\right)\right).$
Assuming a differentiable CDF of $P^{x},$ \citet{RefWorks:1037}
proved the weak consistency of the frequentist posterior median $\vartheta\left(1/2\right)$
and the frequentist posterior mean $\bar{\vartheta}_{x}$ and proved
that the former is median-unbiased. In that case, the frequentist
mode, the value maximizing the probability density function of $\vartheta$,
is also available if a unique maximum exists. 

The \emph{frequentist posterior predictive distribution}, the frequentist
analog of the Bayesian posterior predictive distribution of a new
observation of $X$, is $P^{\left(x\right)}=\int_{\Theta}P_{\vartheta,\gamma}dP^{x}\left(\vartheta\right)$
for all $x\in\Omega$. (\citet{RefWorks:1270}, \citet{RefWorks:1369},
and \citet{RefWorks:1175} considered this with fiducial-like distributions
in place of the confidence measure $P^{x}$.) Appropriate point predictions
are $\bar{\xi}_{x}=\int_{\Omega}X\left(\omega\right)dP^{\left(x\right)}\left(\omega\right)$
in the {}``regression'' case of continuous $\Omega$ and $\tilde{\xi}_{x}=1_{\left[1/2,1\right]}\left(P^{\left(x\right)}\left(X=1\right)\right)$
in the {}``classification'' case in which $\Omega=\left\{ 0,1\right\} $.
If $P^{x}$ is approximated using a bootstrap algorithm as in Section
\ref{sub:Hypothesis}, then the resulting values of $\bar{\xi}_{x}$
and $\tilde{\xi}_{x}$ are bootstrap aggregation (bagging) predictions;
\citet{Breiman1996123} found bagging to reduce prediction error.
The confidence predictive distribution can also be used to determine
sizes of new studies by accounting for uncertainty in the effect size.
(The classical method of determining the sample size of a planned
experiment is often criticized for relying on a point estimate of
the effect size.) 

~

\subsection{\label{sub:The-Bayesian-framework}Confidence versus Bayesian probability}

As the examples of Section \ref{sub:Decision-degenerate} illustrate,
many uses of Bayesian posterior distributions are completely compatible
with confidence measures since both distributions of parameters deliver
coherent inferences in the form of probabilities that hypotheses of
interest are true. However, to the extent that updating parameter
distributions in agreement with valid confidence intervals conflicts
with updating them by Bayes's formula, confidence decision theory
differs fundamentally from the two dominant forms of Bayesianism,
{}``subjective'' Bayesianism, which is seldom used by the statistics
community, and {}``objective'' Bayesianism broadly defined as a
collection of algorithms for generating prior distributions from sampling
distributions or from invariance arguments. Nonetheless, the proposed
framework follows from an application of de Finetti's theory of prevision
to an agent that makes decisions according to certain confidence levels
(§\ref{sub:Coherence}).

\subsubsection{Bayesian conditioning}

As demonstrated in Section \ref{sub:Coherence}, the proposed framework
for frequentist inference satisfies coherence, which does not require
the probability distribution of the parameters to correspond to any\emph{
}Bayesian posterior distribution, a prior distribution conditional
on the observed data in the Kolmogorov sense, as is frequently supposed.
Not coherence but another pillar of Bayesianism mandates that the
posterior distribution, i.e., the parameter distribution used for
decisions after making an observation, must equal the prior distribution
conditioned on the observation \citep{temporal-Goldstein}. That assumption,
usually implicit, has been stated as a plausible principle of learning
from data: 
\begin{defn}[Bayesian temporal principle]
Consider the \emph{prior distribution} $\pi$, a probability measure
induced by a random vector $\vartheta$ in $\Theta$, the parameter
space. Let the \emph{update rule} $\pi_{\bullet}^{\prime}$ denote
a function mapping $\Omega$, the sample space, to a set of probability
measures, each defined on $\Theta$. If, for all $x^{\prime}\in\Omega,$
the \emph{posterior distribution }$\pi_{x^{\prime}}^{\prime}$ induced
by random quantity $\vartheta_{x^{\prime}}^{\prime}$ in $\Theta$
is the conditional distribution of $\vartheta$ given $X^{\prime}=x^{\prime}$,
then $\pi_{\bullet}^{\prime}$ satisfies the \emph{Bayesian temporal
principle,} $\pi_{x^{\prime}}^{\prime}$ is called a \emph{Bayesian
posterior distribution}, and the equivalence between the posterior
and conditional distributions is written as\[
\vartheta_{x^{\prime}}^{\prime}\equiv\vartheta|x^{\prime}.\]
\end{defn}
\begin{rem}
In the one-dimensional case, the Bayesian temporal principle stipulates
that, for all $\Theta^{\prime}\subseteq\Theta$, \[
\pi_{x}^{\prime}\left(\vartheta^{\prime}\in\Theta^{\prime}\right)=\pi\left(\vartheta\in\Theta^{\prime}|X^{\prime}=x^{\prime}\right),\]
where $\pi_{x}^{\prime}$ and $\pi$ are the posterior and prior distributions
of $\vartheta^{\prime}$ and $\vartheta$, respectively. Adding a
prime symbol $\left(^{\prime}\right)$ for each successive observation
gives $\vartheta_{x^{\prime}}^{\prime}\equiv\mathbf{\vartheta}|x^{\prime}$,
$\vartheta_{x^{\prime\prime}}^{\prime\prime}\equiv\vartheta_{x^{\prime}}^{\prime}|x^{\prime\prime}$,
$\vartheta_{x^{\prime\prime\prime}}^{\prime\prime\prime}\equiv\vartheta_{x^{\prime\prime}}^{\prime\prime}|x^{\prime\prime\prime}$,
and so forth. \citet{RefWorks:1244} coined the name of the principle,
explaining that it unreasonably requires that an agent's conditional
betting odds (prior odds conditional on a contemplated future observation)
determines its future betting odds (posterior odds as a function of
the actual observation). In other words, the current rate of machine
learning is limited by the previous strength of machine belief.
\end{rem}
\citet{RefWorks:1244} pointed out that although Bayesians follow
the temporal principle when using Bayes's formula, they disregard
it every time they revise a prior or sampling model upon seeing new
data. Such revision occurs whenever posterior predictions are subjected
to frequentist model checking procedures such as cross validation.
One rationale for revising the prior is that poor frequentist performance
may indicate that it did not adequately reflect the available information
as well as it might have had it been more carefully elicited. Another
is the receipt of new information that cannot be represented in the
probability space of the initial prior \citep{DiaconisZabell1982b}.

\subsubsection{\label{sub:Non-Bayesian-coherence}Non-Bayesian coherence}

Confidence decision theory not only satisfies coherence in the sense
of avoiding sure loss (§\ref{sub:Coherence}), but, when reduced to
the minimization of expected loss with respect to a single confidence
measure (§\ref{sub:Decision-degenerate}), is also coherent in the
sense of axiomatic systems of expected utility maximization \citep{MaxUtility1944,RefWorks:126}.
While both approaches to coherence support the concept of placing
bets in accord with the laws of probability, including conditional
probability for called-off bets, none of the approaches entails the
equality of conditional probability as defined by Kolmogorov and posterior
probability as the hypothesis probability updated as a function of
observed data. Replacing probabilities with proposition truth values
and conditional probabilities with theorems (statements of implication)
furnishes an illustration from deductive logic \citep{RefWorks:1203}:
an agent whose set of propositions held to be true do not contradict
each other at any point in time is completely self-consistent. However,
the agent cannot comply with the deductive version of the Bayesian
temporal principle unless none of the truth values ever requires revision
\citep{ColinHowson12011997}. As a finitely additive probability distribution,
the confidence measure also agrees with axiomatic systems of probabilistic
logic such as that of \citet{RTCox}.

The above accounts of coherence provide no support for the Bayesian
temporal principle since their theorems involve conditional probability,
not posterior probability as specified by some update rule $\pi_{\bullet}^{\prime}$.
Simply defining\emph{ }the posterior distribution to be Kolmogorov's
conditional distribution given the data either specifies nothing about
how parameter distributions are updated with new data or conceals
the assumption of the Bayesian temporal principle \citep{RefWorks:1097}. 

Even though the statistical literature refers to many theorems supporting
coherence and rationality as understood in Section \ref{sub:Coherence},
discussion of the foundational principle of Bayesianism has instead
taken place mostly in the philosophical literature. David Lewis \citep{Teller1973218}
presented a transformation of the Dutch book game (§\ref{sub:Coherence})
into one in which the gambler knows the rule the casino agent uses
to update its betting odds on receipt of new information. In that
game, but not in the original Dutch book game, violation of the Bayesian
temporal principle leads the casino to sure loss \citep{Teller1973218,RefWorks:1258}.
Since such violation occurs over time, it is considered a breach of
\emph{diachronic game-theoretic coherence}, a restriction on the degree
to which an agent's betting odds can change over time, as opposed
to \emph{synchronic game-theoretic coherence}, a consistency in an
agent's betting odds at any given time \citep{Armendt1992}. Accordingly,
the Dutch book arguments for diachronic coherence\emph{ }have been
considered much weaker \citep{RefWorks:1255,RefWorks:Goldstein2006,Williamson20091}
than those for synchronic coherence, the type of coherence supported
by the theorems of \citet{RefWorks:43} and \citet{RefWorks:126}.
\citet{Goldstein199755}, \citet[pp. 256-260]{Hacking2001}, and \citet{Williamson20091},
while accepting Dutch book arguments for synchronic coherence, do
not consider diachronic coherence to be a requirement of logical thought.
\citet{Hild1998b} distinguished game-theoretic diachronic coherence
from  decision-theoretic diachronic coherence, arguing that the latter
rules out the Bayesian temporal principle as incoherent. Another difficulty
is that some Dutch book arguments lead to versions of diachronic coherence
that conflict with the Bayesian temporal principle \citep{Armendt1992}. 

In summary, the theorems routinely presented as proof that all rational
thought or coherent decision making must be Bayesian actually prove
no more than the irrationality of violating the logic of standard
probability theory. Thus, any decision-theoretic framework representing
unknown values as random quantities mapped from some probability space
stands on equal ground with Bayesianism as far as the minimal requirements
of rationality are concerned. Such frameworks include geometric conditioning
\citep{RefWorks:1244}, probability kinematics \citep{DiaconisZabell1982b,Jeffrey2004},
dynamic coherence \citep{RefWorks:1242,RefWorks:1241}, and relative
entropy maximization \citep{Gruenwald20041367,Jaeger2005b,Williamson20091}
as well as confidence decision theory (§\ref{sec:frequentist-framework}).

\subsubsection{Objections to frequentist posteriors}

Since, neglecting sufficiency and ancillarity considerations, the
confidence level is numerically equal to the fiducial probability
in the case of a one-dimensional parameter of interest given continuous
data \citep{Wilkinson1977}, some classical Bayesian objections against
the coherence of fiducial distributions apply with equal force against
the coherence of the confidence measure. The strength of such arguments
is now evaluated in light of the above distinction between axiomatic
coherence and the Bayes update rule. 

In the present framework, confidence-based or fiducial probabilities
of hypotheses correspond to reasonable betting odds, a consequence
that \citet{RefWorks:1187} considered impossible since \citet{RefWorks:1159}
had demonstrated that fiducial distributions are Bayesian posteriors
only in certain special cases and since placing conditional bets contrary
to conditional probability leads to certain loss. The conclusion drawn
by \citet{RefWorks:1187} would only follow under the widely held
but incorrect assumption that a parameter distribution must be a Bayesian
posterior for it to satisfy coherence. \citet{RefWorks:1159}, extending
the work of \citet{RefWorks:1103}, actually had found conditions
under which the fiducial distribution violates the Bayesian temporal
principle considered in Section \ref{sub:The-Bayesian-framework},
not that a conditional fiducial distribution is incompatible with
the definition of a conditional probability distribution. 

\citet{RefWorks:1159} also demonstrated that violation of the Bayesian
temporal principle means the pivot is not unique, leading to non-unique
fiducial distributions. In light of the subsequent failure of a generation
of statisticians to identify any genuinely noninformative priors (\citealp{MarginalizationParadoxes1973};
\citealp[pp. 226-235]{RefWorks:Walley1991}; \citealp{RefWorks:81};
\citealp{RefWorks:1443}), the belated rejoinder is that Bayesian
posteriors lack uniqueness as well \citep{10.1057/9780230226203.0564,RefWorks:1175}.
Just as given a prior, sampling model, and data, all inferences made
using the resulting Bayesian posterior measure are coherent, so given
an exact estimator, sampling model, and data, all inferences made
using the resulting confidence or fiducial measure are equally coherent.
Thus, the selection of frequentist set estimators parallels the selection
of priors, and in each case such selection may depend on the intended
application. Section \ref{para:Strategies} points to reasonable criteria
for such selection.

\subsection{\label{sub:Scalar-parameter}\label{sec:frequentist-framework}Scalar
subparameter case}

The equality between tail probabilities of a confidence measures and
\emph{p}-values will be used to prove a consistency property that
holds under more general conditions for a confidence level than for
a \emph{p}-value as estimators of composite hypothesis truth.

\subsubsection{\label{sub:Cumulative-confidence}Confidence CDF as the \emph{p}-value
function}

If decisions are based on a single confidence measure of a scalar
parameter of interest, then the CDF of that measure is an upper-tailed
\emph{p}-value function. 
\begin{defn}
Consider a function $p^{+}:\Omega\times\Theta\rightarrow\left[0,1\right]$
such that $p^{+}\left(x,\bullet\right)=p_{x}^{+}\left(\bullet\right)$
is a CDF for all $x\in\Omega$ and such that \begin{equation}
P_{\xi}\left(p_{X}^{+}\left(\theta\right)<\alpha\right)=\alpha\label{eq:uniform}\end{equation}
for all $\theta\in\Theta$, $\xi\in\Xi$, and $\alpha\in\left[0,1\right].$
Then, for any $x\in\Omega,$ the map $p_{x}^{+}:\Theta\rightarrow\left[0,1\right]$
is called an \emph{upper-tail p}-value function for $\theta.$ Likewise,
$p_{x}^{-}:\Theta\rightarrow\left[0,1\right]$ is called a \emph{lower-tail
p}-value function if \emph{\begin{equation}
p_{x}^{-}\left(\theta\right)=1-p_{x}^{+}\left(\theta\right)\label{eq:pvalue-functions}\end{equation}
}for all $\theta\in\Theta$ and for all $x\in\Omega.$ 
\end{defn}
Uniformly distributed under the simple null hypothesis that $\theta=\theta^{\prime},$
$p_{x}^{-}\left(\theta^{\prime}\right)$ and $p_{x}^{+}\left(\theta^{\prime}\right)$
are exact \emph{p}-values of one-sided tests. Since equation (\ref{eq:pvalue-functions})
is an isomorphism between the two \emph{p}-value functions, the pair
$\left\langle p_{x}^{-}\left(\theta^{\prime}\right),p_{x}^{+}\left(\theta^{\prime}\right)\right\rangle $
will be called the \emph{p}-value function, either element of which
may be designated by $p_{X}^{\pm}\left(\theta^{\prime}\right).$ The
\emph{two-sided p}-value of the null hypothesis that $\theta$ is
in a central region $\Theta^{\prime}$ of $\Theta$  is \[
p_{x}\left(\Theta^{\prime}\right)=2\sup_{\theta^{\prime}\in\Theta^{\prime}}p_{x}^{-}\left(\theta^{\prime}\right)\wedge p_{x}^{+}\left(\theta^{\prime}\right)\]
for all $x\in\Omega,$ reducing to the usual $p_{x}\left(\Theta^{\prime}\right)=2p_{x}^{-}\left(\theta^{\prime}\right)\wedge p_{x}^{+}\left(\theta^{\prime}\right)$
for the point hypothesis that $\theta=\theta^{\prime}.$ 

While the name \emph{p-value function }used by \citet{RefWorks:60}
has become standard in the scientific literature, \emph{significance
function} is also used in higher-order asymptotics (e.g., \citet{RefWorks:437}).
\citet{Efron19933}, \citet{RefWorks:127}, and \citet{RefWorks:1037}
prefer the term \emph{confidence distribution}, avoided here to clearly
distinguish the \emph{p}-value function from the confidence measure
as a Kolmogorov probability distribution. (Whereas any \emph{p}-value
function is isomorphic to a unique confidence measure as defined in
Section \ref{sub:Confidence-measures-etc}, the \emph{p}-value function
can also be isomorphic to an incomplete probability measure. \citet{Wilkinson1977}
constructed a theory of incoherence based on such a measure, underscoring
the need to sharply distinguish confidence measures from \emph{p}-value
functions.)

By the usual concept of statistical power, the \emph{Type II error
rate }of $p^{\pm}$ associated with testing the false null hypothesis
that $\theta=\theta^{\prime}$ at significance level $\alpha$ is
$\beta^{\pm}\left(\alpha,\theta,\theta^{\prime}\right)=P_{\xi}\left(p_{X}^{\pm}\left(\theta^{\prime}\right)>\alpha\right)$
for any $\theta\gtrless\theta^{\prime}$.  For all $\alpha_{1},\alpha_{2}\in\left[0,1\right]$
such that $\alpha_{1}+\alpha_{2}<1$, \[
P_{\xi}\left(\alpha_{1}<p_{X}^{+}\left(\theta\right)<1-\alpha_{2}\right)=1-\alpha_{1}-\alpha_{2},\]
implying that $\theta_{X}^{+}:\left[0,1\right]\rightarrow\Theta,$
the inverse function of $p_{X}^{+},$ yields $\left(\theta_{X}^{+}\left(\alpha_{1}\right),\theta_{X}^{+}\left(1-\alpha_{2}\right)\right)$
as an exact $100\left(1-\alpha_{1}-\alpha_{2}\right)\%$ confidence
interval \citep{RefWorks:60,Efron19933,RefWorks:127,RefWorks:1037}.
\begin{rem}
\label{rem:asymptotics}In many applications, approximate \emph{p}-value
functions replace those that exactly satisfy the definition. For instance,
\citet{RefWorks:127} use a half-corrected \emph{p}-value function
like $p_{C,x}$ of Example \ref{exa:binomial} for discrete data.
Other approximations involve parameter distributions with asymptotically
correct frequentist coverage, including the asymptotic \emph{p}-value
functions of \citet{RefWorks:130}, the distributions of asymptotic
generalized pivotal quantities of \citet{RefWorks:1092}, some of
the generalized fiducial distributions of \citet{RefWorks:1175},
and the Bayesian posteriors of Section \ref{sub:Background}. As with
frequentist inference in general, asymptotics provide approximations
that in many applications prove sufficiently accurate for inference
in the absence of exact results \citep{RefWorks:1250}. 
\end{rem}

\subsubsection{Interpretations of the \emph{p}-value function}

In its history, the \emph{p}-value function has had Neymanian, Fisherian,
and Bayesian interpretations. Consistently viewing the \emph{p}-value
function within the Neyman-Pearson framework rather than as the CDF
of a probability measure of $\theta,$ \citet{RefWorks:60}, \citet{RefWorks:127},
\citet{RefWorks:130}, and \citet{RefWorks:1037} have used $p^{+}$
to concisely present information about hypothesis tests and confidence
intervals in data analysis results. The \emph{p}-value function thus
interpreted as a warehouse of results of potential hypothesis tests
and confidence intervals has also uncovered relationships with the
Bayesian and fiducial frameworks \citep{RefWorks:127}. \citet{RefWorks:127}
aimed {}``to demonstrate the power of the frequentist methodology''
by means of reporting on the \emph{p}-value function and likelihood
function as key components of a unified Neyman-Pearson alternative
to Bayesian posterior distributions, which can fail to yield interval
estimates guaranteed to cover true parameter value at some given rate.
Interestingly, the incipient \emph{p}-value function had been originally
conceived as a Fisherian alternative to what was seen as a mechanical
use of the Neyman-Pearson confidence interval \citep{conditionalityPrinciple1958}.

In a move away from both of the main frequentist interpretations of
the \emph{p}-value function, \citet{Efron19933} proposed a simple,
fast algorithm for computing an \emph{implied prior density} and an
\emph{implied likelihood} from a confidence density assumed to be
proportional to a Bayesian posterior density. He reported that with
a confidence density based on an exponential model and the ABC confidence
interval method, the disagreement between the implied likelihood and
the true likelihood observed by \citet{RefWorks:1159} {}``is small
in most cases,'' with the implication that the confidence density
approximates a Bayesian posterior, thereby establishing approximate
coherence. However, while compatibility with a Bayesian posterior
is sufficient for coherence, it is by no means necessary (§§\ref{sub:Coherence},
\ref{sub:The-Bayesian-framework}).

Dropping the requirement of approximating a Bayesian posterior enables
more exact frequentist coverage in many instances without sacrificing
the coherence achieved by \citet{Efron19933}. The concept of coherence
is itself sufficient to recast the \emph{p}-value function from a
pure Neyman-Pearson toolbox into a versatile weapon for statistical
inference and decision making, enabling all of the applications available
to a Bayesian posterior distribution of the interest parameter, marginal
over any nuisance parameters \citep[cf. ][]{RefWorks:193}. 

In addition, information in the form of a subjective prior distribution
can be incorporated into frequentist data analysis by combining the
prior with the \emph{p}-value function \citep{RefWorks:24} under
the following circumstances. Suppose Agents A and B each bases the
posterior probability measure by which it makes decisions (§\ref{sub:Decision-degenerate})
on confidence sets according to the framework of Section \ref{sub:Coherence}
whenever the observation that $X=x$ constitutes the only information
about the parameter of interest. Agent A observes $x,$ which would
yield the confidence measure $P^{x}$ on $\left(\Theta,\mathcal{A}\right),$
but it also has independent information in the form of $Q,$ a probability
measure on $\left(\Theta,\mathcal{A}\right)$ elicited from Agent
B, where $\Theta\subseteq\mathbb{R}^{1}.$ Since Agent B would have
set $Q$ to equal a confidence measure if possible, Agent A processes
$Q$ exactly as it would a confidence measure computed on the basis
of data independent of $X.$ Since each of several methods of combining
\emph{p}-value functions from independent data sets yields an approximate
\emph{p}-value function incorporating information from both data sets
\citep{RefWorks:130}, Agent A bases its decisions on $P^{x}\oplus Q,$
the probability measure of the CDF obtained by applying any such combination
method to the CDFs of $P^{x}$ and $Q.$ It follows that if $Q$ is
in fact a confidence measure, then $P^{x}\oplus Q$ is a confidence
measure to the same degree of approximation as the combined CDF is
a \emph{p}-value function. Agents A and B may actually be the same
agent, which would be the case if Agent A had computed the prior $Q$
as a confidence measure on the basis of independent data that are
no longer available. In conclusion, the presence of important information
in the form of a prior probability distribution on $\left(\Theta,\mathcal{A}\right)$
does not in itself necessitate moving from confidence-based statistics
to Bayesian statistics.

\subsubsection{\label{sub:Confidence-vs-p}Confidence levels versus \emph{p}-values}

Although both confidence levels and \emph{p}-values can be computed
from the same \emph{p}-value function, the following examples illustrate
how they can lead to different inferences and decisions. Section \ref{sub:Consistency-of-support}
then demonstrates that the former but not the latter are consistent
as estimators of composite hypothesis truth. 
\begin{example}[point null hypothesis]
\label{exa:continuous}If $P^{x}\left(\vartheta<\bullet\right)$
is continuous on $\Theta$, then $P^{x}\left(\theta=\theta^{\prime}\right)=0$
for any interior point $\theta^{\prime}$ of $\Theta$. This means
that given any alternative hypothesis $\theta\in\Theta^{\prime}$
such that $P^{x}\left(\theta\in\Theta^{\prime}\right)>0$, betting
on $\theta=\theta^{\prime}$ versus $\theta\in\Theta^{\prime}$ at
any finite betting odds will result in expected loss, reflecting the
absence of information singling out the point $\theta=\theta^{\prime}$
as a viable possibility before the data were observed. (By contrast,
the usual two-sided \emph{p}-value is numerically equal to $p_{x}\left(\theta^{\prime}\right),$
which does not necessarily equal the probability of any hypothesis
of interest.) If, on the other hand, the parameter value can equal
the null hypothesis value for all practical purposes, that fact may
be represented by modeling the parameter of interest as a random effect
with nonzero probability at the null hypothesis value. The latter
option would define the confidence measure such that its CDF is a
predictive \emph{p}-value function such as that used by \citet{RefWorks:1091}.
\end{example}
~
\begin{example}[bioequivalence]
\label{exa:equivalence}Regulatory agencies often need an estimate
of $1_{\left[\theta^{\prime}-\Delta,\theta^{\prime}+\Delta\right]}\left(\theta\right),$
the indicator of whether the hypothesis that the continuous parameter
of interest lies within $\Delta$ of $\theta^{\prime}$ for some $\Delta>0;$
a value common in bioequivalence studies is $\Delta=\log\left(125\%\right)$
with $\exp\left(\theta^{\prime}\right)$ as the efficacy of a medical
treatment. For the purpose of deciding whether to approve a new treatment
or a genetically modified crop, estimates provided by companies with
obvious conflicts of interest must be as objective as possible. The
Neyman-Pearson framework in effect enables conservative tests of the
null hypotheses $\theta\in\left[\theta^{\prime}-\Delta,\theta^{\prime}+\Delta\right]$,
$\theta<\theta^{\prime}-\Delta$, and $\theta>\theta^{\prime}+\Delta$
\citep{RefWorks:438} but without guidance on how to use the resulting
\emph{p}-values $p_{x}\left(\theta^{\prime}\right),$ $p_{x}^{+}\left(\theta^{\prime}-\Delta\right),$
and $p_{x}^{-}\left(\theta^{\prime}+\Delta\right)$ to make coherent
decisions, which would instead require estimates of $1_{\left(-\infty,\theta^{\prime}-\Delta\right)}\left(\theta\right),$
$1_{\left[\theta^{\prime}-\Delta,\theta^{\prime}+\Delta\right]}\left(\theta\right),$
and $1_{\left(\theta^{\prime}+\Delta,\infty\right)}\left(\theta\right)$
such that the sum of the estimates is 1. The probabilities $P^{x}\left(\vartheta<\theta^{\prime}-\Delta\right),$
$P^{x}\left(\theta^{\prime}-\Delta\le\vartheta\le\theta^{\prime}+\Delta\right),$
and $P^{x}\left(\vartheta>\theta^{\prime}+\Delta\right)$ qualify
as such estimates without suffering from the subjective or arbitrary
nature of assigning a prior distribution. Due to the coherence of
probabilistic indicator estimators, regulators may simultaneously
consider more complex estimates such as $P^{x}\left(\vartheta>\theta^{\prime}+\Delta|\vartheta\notin\left[\theta^{\prime}-\Delta,\theta^{\prime}+\Delta\right]\right)$,
the probability that the effect size is high given that it is non-negligible,
without the multiplicity concerns that plague Neymanian statistics
(§\ref{sub:Likelihood}). \citet{RefWorks:1037} also compared the
use of observed confidence levels to conventional methods of bioequivalence.
\end{example}
~

\subsubsection{\label{sub:Consistency-of-support}Consistency of hypothesis confidence}

More terminology will be introduced to establish a sense in which
the confidence value but not the \emph{p}-value consistently estimates
the hypothesis indicator.
\begin{defn}
\label{def:consistency}An indicator estimator $\hat{1}$ is \emph{consistent}
if, for all $\Theta^{\prime}\in\mathcal{A},$ \[
\hat{1}_{\Theta^{\prime}}\left(X\right)\xrightarrow{P_{\theta,\gamma}}1_{\Theta^{\prime}}\left(\theta\right)\]
for every $\gamma\in\Gamma$ and for every $\theta$ that is an element
of $\Theta$ but not of the boundary of $\Theta^{\prime}.$ 
\end{defn}
By the usual concept of statistical power, the \emph{Type II error
rate }of $p^{\pm}$ associated with testing the false null hypothesis
that $\theta=\theta^{\prime}$ at significance level $\alpha$ is
$\beta^{\pm}\left(\alpha,\theta,\theta^{\prime}\right)=P_{\theta,\gamma}\left(p_{X}^{\pm}\left(\theta^{\prime}\right)>\alpha\right)$
for any $\theta\gtrless\theta^{\prime}$. Commonly used in two-sided
testing, the \emph{two-sided p}-value of the null hypothesis that
$\theta\in\Theta^{\prime}$ is  for all $\Theta^{\prime}\subseteq\Theta$
and $x\in\Omega$. 

The next two propositions contrast the consistency of the confidence
value with the inconsistency of the two-sided \emph{p}-value. 
\begin{prop}
\label{thm:congruity-is-consistent}Assume all one-sided tests represented
by the \emph{p}-value functions $p^{\pm}$ are asymptotically powerful
in the sense that $\lim_{n\rightarrow\infty}\beta^{\pm}\left(\alpha,\theta,\theta^{\prime}\right)=0$
for all $\alpha\in\left(0,1\right)$ and for all $\theta,\theta^{\prime}\in\Theta$
such that $\theta\gtrless\theta^{\prime}.$ The function $\hat{1}:\mathcal{A}\times\Omega\rightarrow\left[0,1\right]$
is a consistent indicator estimator if $P^{x}=\hat{1}_{\bullet}\left(x\right)$
is a confidence measure corresponding to $p^{\pm}$ given $X=x$ for
all $x\in\Omega.$\end{prop}
\begin{proof}
By the definition of the boundary of a set $\Theta^{\prime}$ as the
difference between its closure $\bar{\Theta^{\prime}}$ and its interior
$\interior\Theta^{\prime}$, the theorem asserts that, for all $\Theta^{\prime}\in\mathcal{A},$
$\theta$ is either in $\interior\Theta^{\prime},$ in which case
the theorem asserts $P^{X}\left(\Theta^{\prime}\right)\xrightarrow{P_{\theta,\gamma}}1,$
or $\theta$ is in $\Theta\backslash\Theta^{\prime},$ in which case
the theorem asserts $P^{X}\left(\Theta^{\prime}\right)\xrightarrow{P_{\theta,\gamma}}0.$
Let $\mathcal{A}^{\prime}$ represent the set of all disjoint open
Each term of the sum expands as

\begin{eqnarray*}
P^{X}\left(\Theta^{\prime\prime\prime}\right) & = & P^{X}\left(\left(\inf\Theta^{\prime\prime\prime},\sup\Theta^{\prime\prime\prime}\right)\right)=p_{X}^{+}\left(\sup\Theta^{\prime\prime\prime}\right)-p_{X}^{+}\left(\inf\Theta^{\prime\prime\prime}\right)\\
 & = & p_{X}^{-}\left(\inf\Theta^{\prime\prime\prime}\right)-p_{X}^{-}\left(\sup\Theta^{\prime\prime\prime}\right)\\
 & = & 1-p_{X}^{-}\left(\sup\Theta^{\prime\prime\prime}\right)-p_{X}^{+}\left(\inf\Theta^{\prime\prime\prime}\right).\end{eqnarray*}
As the \emph{p}-value functions are asymptotically powerful, $p_{X}^{\pm}\left(\theta^{\prime}\right)\xrightarrow{P_{\theta,\gamma}}0$
for all $\alpha\in\left(0,1\right)$ and for all $\theta,\theta^{\prime}\in\Theta$
such that $\theta\gtrless\theta^{\prime},$ with the result that each
term may be written as a function of \emph{p}-values that converge
in $P_{\theta,\gamma}$ to 0:\begin{eqnarray*}
P^{X}\left(\Theta^{\prime\prime\prime}\right) & = & \begin{cases}
p_{X}^{-}\left(\inf\Theta^{\prime\prime\prime}\right)-p_{X}^{-}\left(\sup\Theta^{\prime\prime\prime}\right) & \theta<\inf\Theta^{\prime\prime\prime}\\
1-p_{X}^{-}\left(\sup\Theta^{\prime\prime\prime}\right)-p_{X}^{+}\left(\inf\Theta^{\prime\prime\prime}\right) & \theta\in\Theta^{\prime\prime\prime}\\
p_{X}^{+}\left(\sup\Theta^{\prime\prime\prime}\right)-p_{X}^{+}\left(\inf\Theta^{\prime\prime\prime}\right) & \theta>\sup\Theta^{\prime\prime\prime}\end{cases}\\
 & \xrightarrow{P_{\theta,\gamma}} & \begin{cases}
0-0 & \theta<\inf\Theta^{\prime\prime\prime}\\
1-0-0 & \theta\in\Theta^{\prime\prime\prime}\\
0-0 & \theta>\sup\Theta^{\prime\prime\prime}\end{cases}\end{eqnarray*}
for all $\Theta^{\prime\prime\prime}\in\mathcal{A}^{\prime}$. Summing
the terms over $\mathcal{A}^{\prime}$ yields \[
P^{X}\left(\Theta^{\prime}\right)\xrightarrow{P_{\theta,\gamma}}\sum_{\Theta^{\prime\prime\prime}\in\mathcal{A}^{\prime}}1_{\Theta^{\prime\prime\prime}}\left(\theta\right)=1_{\Theta^{\prime}}\left(\theta\right)\]
since $\theta\in\interior\Theta^{\prime}$ implies that $\theta$
is in one element of $\mathcal{A}^{\prime}$.\end{proof}
\begin{rem}
\citet[pp. 37-38]{Polansky2007b} proved a similar proposition of
consistency given a smooth distribution $P_{\theta,\gamma}.$ A suitably
transformed likelihood ratio test statistic is also a consistent indicator
estimator under the standard regularity conditions \citep{RefWorks:435}.\end{rem}
\begin{prop}
Under the conditions of Theorem \ref{thm:congruity-is-consistent},
the two-sided \emph{p}-value $p_{X}\left(\Theta^{\prime}\right)$
is not a consistent indicator estimator. \end{prop}
\begin{proof}
For any $\theta\in\Theta^{\prime}\in\mathcal{A},$ the distribution
of the two-sided \emph{p}-value $p_{X}\left(\Theta^{\prime}\right)$
converges to the uniform distribution on $\left[0,1\right]$ \citep{RefWorks:1037},
violating consistency (Definition \ref{def:consistency}).
\end{proof}

\section{\label{sec:Discussion}Discussion}

The confidence metameasure $\mathcal{P}^{x}$ and the confidence measure
or frequentist posterior $P^{x}$ bring both coherence and consistency
to frequentist inference and decision making.

The coherence property established in Section \ref{sub:Coherence}
confers the ability to consistently and directly report the levels
of confidence of as many complex hypotheses as desired and to perform
estimation and prediction (§\ref{sub:Decision-degenerate}). Even
though the frequentist posterior $P^{x}$ is a flexible distribution
of possible values of a fixed parameter, it requires no prior; in
fact, $P^{x}$ need not even necessarily correspond to any Bayesian
posterior distribution (§\ref{sub:The-Bayesian-framework}). In conclusion,
the metalevel or level of confidence in a given hypothesis has the
internal coherence of the Bayesian posterior or class of such posteriors
without requiring a prior distribution or even an exact confidence
set estimator. 

More can be said if the parameter of interest is one-dimensional,
in which case the confidence level of a composite hypothesis is consistent
as an estimate of whether that hypothesis is true, whereas neither
the Bayesian posterior probability nor the \emph{p}-value is generally
consistent in that sense (§\ref{sub:Consistency-of-support}). Specifically,
the equality of the confidence level of $\theta\in\Theta^{\prime}$
to the coverage rate of the corresponding confidence set guarantees
convergence in probability to 1 if $\theta$ is in the interior of
$\Theta^{\prime}$ or to 0 if $\theta\notin\Theta^{\prime}$ (Proposition
\ref{thm:congruity-is-consistent}).

\section{Acknowledgments}

Anthony Davison furnished many useful comments on an early version
of the manuscript that led to greater generality and clarity. I also
thank Michael Goldstein for information that fortified the discussion
of coherence in Section \ref{sub:The-Bayesian-framework} and Corey
Yanofsky for providing insightful feedback on a draft of the same
section. This work was partially supported by the Faculty of Medicine
of the University of Ottawa and by Agriculture and Agri-Food Canada.

\bibliographystyle{elsarticle-harv}
\bibliography{refman}

\end{document}